\documentclass[reqno]{amsart}

\usepackage{amssymb}
\usepackage{graphicx}
\usepackage{amscd}
\usepackage[hidelinks]{hyperref}
\usepackage{color}
\usepackage{float}
\usepackage{graphics,amsmath,amssymb}
\usepackage{amsthm}
\usepackage{amsfonts}
\usepackage{latexsym}
\usepackage{epsf}
\usepackage{xifthen}
\usepackage{mathrsfs}
\usepackage{dsfont}
\usepackage{makecell}
\usepackage[FIGTOPCAP]{subfigure}
\usepackage{amsmath}
\allowdisplaybreaks[4]
\usepackage{listings}
\usepackage{etoolbox}
\usepackage{fancyhdr}
\usepackage{pdflscape}
\usepackage[title,toc,titletoc]{appendix}
\usepackage{enumitem}
\usepackage[noadjust]{cite}

 \newtheoremstyle{mytheorem}
 {3pt}
 {3pt}
 {\slshape}
 {}
 {\bfseries}
 {.}
 { }
 {}

\numberwithin{equation}{section}

\theoremstyle{theorem}
\newtheorem{theorem}{Theorem}[section]
\newtheorem*{theorem*}{Theorem}

\newtheorem{lemma}[theorem]{Lemma}

\theoremstyle{definition}

\newtheorem*{example*}{Example}

\theoremstyle{remark}
\newtheorem{remark}{Remark}[section]
\newtheorem*{remark*}{Remark}
\newtheorem*{remarks*}{Remarks}

\newcommand{\doi}[1]{\href{http://dx.doi.org/#1}{doi: #1}}

\newcommand{\Keywords}[1]{\ifthenelse{\isempty{#1}}{}{\smallskip \smallskip \noindent \textbf{Keywords}. #1}}
\newcommand{\MSC}[2][2010]{\ifthenelse{\isempty{#2}}{}{\smallskip \smallskip \noindent \textbf{#1MSC}. #2}}
\newcommand{\abstractnote}[1]{\ifthenelse{\isempty{#1}}{}{\smallskip \smallskip \noindent \textsuperscript{\dag}#1}}

\makeatletter
\def\specialsection{\@startsection{section}{1}%
  \z@{\linespacing\@plus\linespacing}{.5\linespacing}%
  {\normalfont}}
\def\section{\@startsection{section}{1}%
  \z@{.7\linespacing\@plus\linespacing}{.5\linespacing}%
  {\normalfont\scshape}}
\patchcmd{\@settitle}{\uppercasenonmath\@title}{\Large\boldmath}{}{}
\patchcmd{\@settitle}{\begin{center}}{\begin{flushleft}}{}{}
\patchcmd{\@settitle}{\end{center}}{\end{flushleft}}{}{}
\patchcmd{\@setauthors}{\MakeUppercase}{\normalsize}{}{}
\patchcmd{\@setauthors}{\centering}{\raggedright}{}{}
\patchcmd{\section}{\scshape}{\large\bfseries\boldmath}{}{}
\patchcmd{\subsection}{\bfseries}{\bfseries\boldmath}{}{}
\renewcommand{\@secnumfont}{\bfseries}
\patchcmd{\@startsection}{\@afterindenttrue}{\@afterindentfalse}{}{}
\patchcmd{\abstract}{\leftmargin3pc}{\leftmargin1pc}{}{}

\def\maketitle{\par
  \@topnum\z@ 
  \@setcopyright
  \thispagestyle{empty}
  \ifx\@empty\shortauthors \let\shortauthors\shorttitle
  \else \andify\shortauthors
  \fi
  \@maketitle@hook
  \begingroup
  \@maketitle
  \toks@\@xp{\shortauthors}\@temptokena\@xp{\shorttitle}%
  \toks4{\def\\{ \ignorespaces}}
  \edef\@tempa{%
    \@nx\markboth{\the\toks4
      \@nx\MakeUppercase{\the\toks@}}{\the\@temptokena}}%
  \@tempa
  \endgroup
  \c@footnote\z@
  \@cleartopmattertags
}
\makeatother

\newcommand{\fl}[1]{\left\lfloor#1\right\rfloor}

\title[1-Shell TSPPs modulo powers of $5$]{1-Shell totally symmetric plane partitions (TSPPs) modulo powers of $5$}

\author[S. Chern]{Shane Chern}
\address{Department of Mathematics, Penn State University, University Park, PA 16802, USA}
\email{shanechern@psu.edu}

\date{}

\begin{document}

%

\maketitle

\begin{abstract}

Let $s(n)$ be the number of 1-shell totally symmetric plane partitions (TSPPs) of $n$. In this paper, an infinite family of congruences modulo powers of $5$ for $s(n)$ will be deduced through an elementary approach. Namely,
$$s\left(2\cdot 5^{2\alpha-1}n+5^{2\alpha-1}\right)\equiv 0 \pmod{5^{\alpha}}.$$


\MSC{11P83, 05A17.}
\end{abstract}

\section{Introduction}

The notion of \textit{1-shell totally symmetric plane partitions}, or \textit{1-shell TSPPs}, was introduced by Blecher \cite{Ble2012}. Here a 1-shell TSPP is a plane partition with a self-conjugate first row/column (as an ordinary partition) and all other entries being $1$.
We have examples like
$$
\begin{array}{ccc}
3 & 2 & 1\\
2 & 1 & 1\\
1 & 1
\end{array}
\quad\text{and}\quad
\begin{array}{cccc}
4 & 4 & 2 & 2 \\
4 & 1 & 1\\
2 & 1 & 1\\
2
\end{array}.$$
Let $s(n)$ denote the number of 1-shell TSPPs of $n$. The following generating function identity was obtained by Blecher \cite{Ble2012}:
\begin{equation}
\sum_{n\ge 0}s(n)q^n=1+\sum_{n\ge 1}q^{3n-2}\prod_{i=0}^{n-2}\left(1+q^{6i+3}\right).
\end{equation}
In \cite{HS2014}, Hirschhorn and Sellers further proved that $s(n)=0$ if $n\equiv 0,2 \pmod{3}$ for $n\ge 1$. It was also shown by them that for $n\ge 0$,
\begin{equation}\label{eq:fg}
s(6n+1)=g(n),
\end{equation}
where $g(n)$ is defined by
\begin{equation}\label{eq:gn}
\sum_{n\ge 0}g(n)q^n=\frac{(q^2;q^2)_{\infty}^3}{(q;q)_{\infty}^2}.
\end{equation}
Throughout, we adopt the standard notation
$$(a;q)_\infty := \prod_{n\ge 0}(1-aq^n).$$

Like other partition functions, the arithmetic properties of $s(n)$ were extensively studied as well. In particular, congruences modulo small powers of $5$ include:
\begin{gather}
s(10n+5)\equiv 0 \pmod{5},\label{eq:cong-hir-sel}\\
s(250n+125)\equiv 0 \pmod{25}\label{eq:cong-xia}\\
\intertext{and}
s(1250n+125,\ 1125)\equiv 0 \pmod{125},\label{eq:cong-chern}
\end{gather}
the first of which was shown by Hirschhorn and Sellers \cite{HS2014}, the second of which was shown by Xia \cite{Xia2015}, and the last of which was due to the author \cite{Che2017}. From these instances, it is natural to expect an infinite family of congruences modulo arbitrary powers of $5$.

In this paper, along with other separate congruences, we shall settle this problem.

\begin{theorem}\label{th:main}
	For $\alpha\ge 1$ and $n\ge 0$, we have
	\begin{equation}\label{eq:cong-main}
	s\left(2\cdot 5^{2\alpha-1}n+5^{2\alpha-1}\right)\equiv 0 \pmod{5^{\alpha}}.
	\end{equation}
\end{theorem}

\begin{remark}
	Taking $\alpha=1$ and $2$, respectively, in \eqref{eq:cong-main}, one may recover \eqref{eq:cong-hir-sel} and \eqref{eq:cong-xia}. To obatin \eqref{eq:cong-chern}, we need one more step of 5-dissection; the details will be discussed in Section \ref{Sect:more-cong}.
\end{remark}

Instead of using the well-known approach of Gordon and Hughes \cite{GH1981}, which relies on modular forms, we will give an elementary proof of Theorem \ref{th:main}. Our idea is similar to that in the work of Hirschhorn and the author \cite{CH2019}.

\section{Fundamental relations}

We require two auxiliary functions:
\begin{gather*}
\xi=\xi(q):=q^{-4}\frac{E(q^2)^3 E(q^{25})^2}{E(q)^2 E(q^{50})^3}\\
\intertext{and}
X=X(q):=\frac{E(q^2)^2 E(q^5)^4}{E(q)^4 E(q^{10})^2},
\end{gather*}
where $E(q):=(q;q)_\infty$. Also, the $U$-operator is defined by
$$U\Bigg(\sum_{n\ge n_0} a_n q^n\Bigg)=\sum_{5n\ge n_0} a_{5n}q^n.$$

\subsection{Initial relations}

We begin with several initial relations.

\begin{align}
U(X)&=11X -60X^2+ 175X^3 -250X^4+ 125X^5,
\label{eq:U-X}\\
U(X^2)&=-24X+ 675X^2 -6700X^3+ 37300X^4 -132500X^5\notag\\
&+ 316875X^6 -512500X^7+ 531250X^8 -312500X^9+ 78125 X^{10},
\\
U(X^3)&=21X -2010X^2+ 49865X^3 -615750X^4+ 4744125X^5\notag\\
& -25301250X^6+ 98718750X^7 -290062500X^8+ 649187500X^9\notag\\
& -1103906250X^{10}+ 1401562500X^{11} -1281250000X^{12}\notag\\
& + 791015625X^{13}-292968750X^{14}+ 48828125X^{15},
\\
U(X^4)&=-8 X + 2984 X^2 - 164200 X^3 + 3899475 X^4\notag\\
& - 54212000 X^5 + 
508067500 X^6 - 3469550000 X^7\notag\\
& + 18095862500 X^8 - 74240000000 X^9 + 
243974062500 X^{10}\notag\\
& - 648743750000 X^{11} + 1400732421875 X^{12} - 
2449687500000 X^{13}\notag\\
& + 3439882812500 X^{14} - 3816406250000 X^{15} + 
3260253906250 X^{16}\notag\\
& - 2060546875000 X^{17} + 903320312500 X^{18} - 
244140625000 X^{19}\notag\\
& + 30517578125 X^{20}
\\
\intertext{and}
U(X^5)&=X - 2650 X^2 + 316275 X^3\notag\\
& - 13553000 X^4 + 314189375 X^5\notag\\
& - 
4710706250 X^6 + 50353584375 X^7\notag\\
& - 407308906250 X^8 + 
2592548671875 X^9\notag\\
& - 13334386718750 X^{10} + 56443662109375 X^{11}\notag\\
& - 
199097265625000 X^{12} + 589929902343750 X^{13}\notag\\
& - 
1474730957031250 X^{14} + 3113805664062500 X^{15}\notag\\
& - 
5542136230468750 X^{16} + 8271881103515625 X^{17}\notag\\
& - 
10261718750000000 X^{18} + 10441009521484375 X^{19}\notag\\
& - 
8547973632812500 X^{20} + 5478668212890625 X^{21}\notag\\
& - 
2639770507812500 X^{22} + 896453857421875 X^{23}\notag\\
& - 
190734863281250 X^{24} + 19073486328125 X^{25}.
\end{align}

Also,
\begin{align}
U(\xi)&=5X,
\label{eq:U-xi}\\
U(\xi X)&=-20 X + 305 X^2 - 1475 X^3 + 3875 X^4 - 5625 X^5 + 3125 X^6,
\\
U(\xi X^2)&=24 X - 1460 X^2 + 23800 X^3 - 191375 X^4 + 948125 X^5\notag\\
& - 3156250 X^6 + 
7284375 X^7 - 11625000 X^8 + 12187500 X^9\notag\\
& - 7421875 X^{10} + 
1953125 X^{11},
\\
U(\xi X^3)&=-9 X + 2705 X^2 - 109575 X^3 + 1915000 X^4\notag\\
& - 19542500 X^5 + 
133646875 X^6 - 658471875 X^7\notag\\
& + 2433656250 X^8 - 6892656250 X^9 + 
15069921875 X^{10}\notag\\
& - 25314453125 X^{11} + 32082031250 X^{12} - 
29580078125 X^{13}\notag\\
& + 18603515625 X^{14} - 7080078125 X^{15} + 
1220703125 X^{16}
\\
\intertext{and}
U(\xi X^4)&=X - 2685 X^2 + 254975 X^3 - 8602875 X^4\notag\\
& + 156808125 X^5 - 
1845540625 X^6 + 15434215625 X^7\notag\\
& - 97122593750 X^8 + 
476838593750 X^9 - 1869609765625 X^{10}\notag\\
& + 5939105859375 X^{11} - 
15403876953125 X^{12} + 32685605468750 X^{13}\notag\\
& - 56548388671875 X^{14} + 
79031005859375 X^{15} - 87777099609375 X^{16}\notag\\
& + 75495605468750 X^{17} - 
48303222656250 X^{18} + 21545410156250 X^{19}\notag\\
& - 5950927734375 X^{20} + 
762939453125 X^{21}.
\end{align}

Here we only prove \eqref{eq:U-X} and \eqref{eq:U-xi}. The rest can be demonstrated analogously but with lengthier calculation.

Recall some preliminary results. Let
$$
R(q)=\left (\begin{matrix} q,q^4\\q^2,q^3\end{matrix};q^5\right )_\infty.
$$
Then (\cite[(8.1.1)]{Hir2017})
\begin{align}\label{eq:E(q)-dis}
E(q)=E(q^{25})\big (R(q^5)^{-1}-q-q^2R(q^5)\big )
\end{align}
and (\cite[(8.4.4)]{Hir2017})
\begin{align}\label{eq:1/E(q)-dis}
\frac1{E(q)}=\frac{E(q^{25})^5}{E(q^5)^6}
&\big (R(q^5)^{-4}+qR(q^5)^{-3}+2q^2R(q^5)^{-2}+3q^3R(q^5)^{-1} \nonumber\\
&+5q^4-3q^5R(q^5)+2q^6R(q^5)^2-q^7R(q^5)^3+q^8R(q^5)^4\big ).
\end{align}
Let (note that at this place we have an extra factor of $q^{-1}$ comparing with the $K$ defined in \cite{CH2019})
$$K=q^{-1}\frac{E(q^2)E(q^5)^5}{E(q)E(q^{10})^5}.$$
Then (\cite[(9.10)]{CH2019})
\begin{align}\label{eq:K+1}
K+1=q^{-1}\frac{E(q^2)^4E(q^5)^2}{E(q)^2E(q^{10})^4},
\end{align}
(\cite[(9.12)]{CH2019})
\begin{align}
K-4=q^{-1}\frac{E(q)^3E(q^5)}{E(q^2)E(q^{10})^3}
\end{align}
and (\cite[(9.11)]{CH2019})
\begin{align}
X=\frac{E(q^2)^2 E(q^5)^4}{E(q)^4 E(q^{10})^2}=\frac{K}{K-4}.
\end{align}
Further, for $\alpha\in\mathbb{Z}_{\ge 0}$ and $\beta\in\mathbb{Z}$, we define
\begin{equation}\label{eq:P-def}
P(\alpha,\beta):=\frac{1}{q^{\alpha}R(q)^{\alpha+2\beta}R(q^2)^{2\alpha-\beta}}+(-1)^{\alpha+\beta}q^{\alpha}R(q)^{\alpha+2\beta}R(q^2)^{2\alpha-\beta}.
\end{equation}
The following results were shown by Tang and the author \cite{CT2019} (see also \cite{CH2019} in which the definition of $P(\alpha,\beta)$ is slightly different).
\begin{gather}
P(\alpha,\beta+1)=4K^{-1} P(\alpha,\beta)+P(\alpha,\beta-1)\label{eq:rec-sqr-1}\\
\intertext{and}
P(\alpha+2,\beta)=K P(\alpha+1,\beta)+ P(\alpha,\beta),\label{eq:rec-sqr-2}
\end{gather}
where
\begin{gather}
P(0,0)=2,\label{eq:p00}\\[3pt]
P(0,1)=\frac{R(q^2)}{R(q)^2}-\frac{R(q)^2}{R(q^2)} =4K^{-1},\label{eq:p01}\\[3pt]
P(1,0)=\frac{1}{qR(q)R(q^2)^2}-q R(q)R(q^2)^2 =K\label{eq:p10}\\
\intertext{and}
P(1,-1)=\frac{R(q)}{qR(q^2)^3}+\frac{q R(q^2)^3}{R(q)} =4K^{-1}-2+K.\label{eq:p1-1}
\end{gather}
Hence, for any $\alpha\in\mathbb{Z}_{\ge 0}$ and $\beta\in\mathbb{Z}$, $P(\alpha,\beta)$ can be represented in $\mathbb{Z}[K,K^{-1}]$ by the above recurrences.

\begin{remark}
In a private communication, Mike Hirschhorn \cite{Hir2019} further transformed \eqref{eq:rec-sqr-1}--\eqref{eq:p1-1} into a beautiful bivariate generating function for $P(\alpha,\beta)$:
\begin{align}\label{eq:Hir-gf-1}
\sum_{\alpha,\beta\ge 0}P(\alpha,\beta)x^\alpha y^\beta=\frac{2-Kx-4K^{-1}y+(K+2+4K^{-1})xy}{(1-Kx-x^2)(1-4K^{-1}y-y^2)}.
\end{align}
Analogously, one could obtain
\begin{align}\label{eq:Hir-gf-2}
\sum_{\alpha,\beta\ge 0}P(\alpha,-\beta)x^\alpha y^\beta=\frac{2-Kx+4K^{-1}y+(K-2+4K^{-1})xy}{(1-Kx-x^2)(1+4K^{-1}y-y^2)}.
\end{align}
\end{remark}

Now let us proceed with the proofs of \eqref{eq:U-X} and \eqref{eq:U-xi}. First,
\begin{align}
U(X)&=\frac{E(q)^4}{E(q^2)^2}U\Big(E(q^{50})^2\big (R(q^{10})^{-1}-q^2-q^4R(q^{10})\big )^2\notag\\
&\quad\quad\quad\quad\times \left(\frac{E(q^{25})^5}{E(q^5)^6}\right)^4
\big (R(q^5)^{-4}+qR(q^5)^{-3}+2q^2R(q^5)^{-2}+3q^3R(q^5)^{-1} \notag\\
&\quad\quad\quad\quad+5q^4-3q^5R(q^5)+2q^6R(q^5)^2-q^7R(q^5)^3+q^8R(q^5)^4\big )^4\Big)\notag\\
\label{eq:2.1-s1}\\
&=q^4\frac{E(q^5)^{20} E(q^{10})^2}{E(q)^{20} E(q^2)^2}\big(P(4,6)-4P(3,6)-80P(3,5)+212P(3,4)\notag\\
&\quad\quad\quad\quad\quad\quad+14P(2,6)+210P(2,5)-418P(2,4)-2200P(2,3)\notag\\
&\quad\quad\quad\quad\quad\quad+1860P(2,2)+708P(1,4)+2800P(1,3)\notag\\
&\quad\quad\quad\quad\quad\quad-1840P(1,2)-2400P(1,1)+60P(1,0)\notag\\
&\quad\quad\quad\quad\quad\quad+1815P(0,2)+3000P(0,1)-1015\big)\label{eq:2.1-s2}\\
&=q^4\frac{E(q^5)^{20} E(q^{10})^2}{E(q)^{20} E(q^2)^2}\frac{(K+1)^4 (K-4)^2 }{K^6}\notag\\
&\quad\times (K^4+144K^3+ 976 K^2 + 1024 K+2816)\label{eq:2.1-s3}\\
&=q^4\frac{E(q^5)^{20} E(q^{10})^2}{E(q)^{20} E(q^2)^2}\frac{(K+1)^4 (K-4)^7}{K^7}\Bigg(11 \left(\frac{K}{K-4}\right)-60\left(\frac{K}{K-4}\right)^2\notag\\
&\quad+175\left(\frac{K}{K-4}\right)^3-250\left(\frac{K}{K-4}\right)^4+125\left(\frac{K}{K-4}\right)^5\Bigg)\notag\\
&=q^4\frac{E(q^5)^{20} E(q^{10})^2}{E(q)^{20} E(q^2)^2}\left(q^{-1}\frac{E(q^2)^4E(q^5)^2}{E(q)^2E(q^{10})^4}\right)^4 \left(q^{-1}\frac{E(q)^3E(q^5)}{E(q^2)E(q^{10})^3}\right)^7\notag\\ &\quad\times\left(q^{-1}\frac{E(q^2)E(q^5)^5}{E(q)E(q^{10})^5}\right)^{-7}\big(11X -60X^2+ 175X^3 -250X^4+ 125X^5\big)\notag\\
&=11X -60X^2+ 175X^3 -250X^4+ 125X^5.\notag
\end{align}
We arrive at \eqref{eq:U-X}. Also,
\begin{align}
U(\xi)&=\frac{E(q^5)^2}{E(q^{10})^3}U\Big(q^{-4}E(q^{50})^3\big (R(q^{10})^{-1}-q^2-q^4R(q^{10})\big )^3\notag\\
&\quad\quad\quad\quad\times \left(\frac{E(q^{25})^5}{E(q^5)^6}\right)^2
\big (R(q^5)^{-4}+qR(q^5)^{-3}+2q^2R(q^5)^{-2}+3q^3R(q^5)^{-1} \notag\\
&\quad\quad\quad\quad+5q^4-3q^5R(q^5)+2q^6R(q^5)^2-q^7R(q^5)^3+q^8R(q^5)^4\big )^2\Big)\notag\\
\label{eq:2.6-s1}\\
&=q^2\frac{E(q^5)^{12}}{E(q)^{12}}\big({-15}P(2,2)+20P(2,1)+50P(1,2)-60P(1,0)\notag\\
&\quad\quad\quad\quad\quad\quad-20P(1,-1)-5P(0,3)-60P(0,2)+75\big)\label{eq:2.6-s2}\\
&=q^2\frac{E(q^5)^{12}}{E(q)^{12}}\frac{5(K+1)^2 (K-4)^3}{K^3}\label{eq:2.6-s3}\\
&=q^2\frac{E(q^5)^{12}}{E(q)^{12}}\frac{(K+1)^2 (K-4)^4}{K^4}\Bigg(5\left(\frac{K}{K-4}\right)\Bigg)\notag\\
&=q^2\frac{E(q^5)^{12}}{E(q)^{12}}\left(q^{-1}\frac{E(q^2)^4E(q^5)^2}{E(q)^2E(q^{10})^4}\right)^2 \left(q^{-1}\frac{E(q)^3E(q^5)}{E(q^2)E(q^{10})^3}\right)^4\notag\\ &\quad\times\left(q^{-1}\frac{E(q^2)E(q^5)^5}{E(q)E(q^{10})^5}\right)^{-4}(5X)\notag\\
&=5X.\notag
\end{align}
This completes the proof of \eqref{eq:U-xi}.

Notably, the derivation of \eqref{eq:2.1-s2}, \eqref{eq:2.1-s3}, \eqref{eq:2.6-s2} and \eqref{eq:2.6-s3} is somewhat intricate. Hence, it is helpful to give more details. I will use \eqref{eq:2.6-s2} and \eqref{eq:2.6-s3} as instances; one should be able to transplant this procedure to \eqref{eq:2.1-s2} and \eqref{eq:2.1-s3} with no difficulty. Also, a simple \textit{Mathematica} program is implemented to verify related calculations. I am happy to supply the codes upon request.

From \eqref{eq:2.6-s1},
\begin{align*}
U(\xi)&= \frac{E(q^5)^{12}}{E(q)^{12}}U\Big(q^{-4}\big (R(q^{10})^{-1}-q^2-q^4R(q^{10})\big )^3\\
&\quad\quad\quad\quad\quad\quad\times \big (R(q^5)^{-4}+qR(q^5)^{-3}+2q^2R(q^5)^{-2}+3q^3R(q^5)^{-1} \notag\\
&\quad\quad\quad\quad\quad\quad\quad+5q^4-3q^5R(q^5)+2q^6R(q^5)^2-q^7R(q^5)^3+q^8R(q^5)^4\big )^2\Big)\\
& =: \frac{E(q^5)^{12}}{E(q)^{12}}U\big(\Pi(q)\big).
\end{align*}
Expanding $\Pi(q)$, selecting terms in which the power of $q$ is a multiple of $5$ and replacing $q^5$ by $q$ (that is, applying the $U$-operator to $\Pi(q)$), one has
\begin{align*}
U\big(\Pi(q)\big)& = - \frac{15}{R(q)^6 R(q^2)^2} + \frac{20}{R(q)^4 R(q^2)^3} + q \frac{50}{R(q)^5} + \cdots + 75 q^2\\[6pt]
&\quad + \cdots -50 q^3 R(q)^5 - 20 q^4 R(q)^4 R(q^2)^3 - 15 q^4 R(q)^6 R(q^2)^2.
\end{align*}
Grouping the first and last terms gives $-15 q^2 P(2,2)$ by the definition \eqref{eq:P-def}. Likewise, we may obtain $20q^2P(2,1)$, $50q^2P(1,2)$ and all other terms in \eqref{eq:2.1-s2}, and therefore complete the derivation.

To arrive at \eqref{eq:2.1-s3}, one may carry out a careful implementation of the recurrences \eqref{eq:rec-sqr-1} and \eqref{eq:rec-sqr-2} for $P(\alpha,\beta)$ so that each $P(,)$ term in \eqref{eq:2.1-s2} is expressed in terms of $K$:
\begin{align*}
P(2,2)&=16K^{-2}+16K^{-1}+10+4K+K^2,\\
P(2,1)&=4K^{-1}+4+2K+K^2,\\
P(1,2)&=16K^{-2}+8K^{-1}+4+K,\\
P(1,0)&=K,\\
P(1,-1)&=4 K^{-1}-2+K,\\
P(0,3)&=64K^{-3}+ 12 K^{-1},\\
P(0,2)&=16 K^{-2}+2.
\end{align*}
Of course, a direct application of the ``series expansion'' command in most computer algebra systems to \eqref{eq:Hir-gf-1} and \eqref{eq:Hir-gf-2} will make our calculation much easier.

\subsection{General relations}

To state the general relations, let us define two infinite matrices $(a_{i,j})_{i\ge 1, j\ge 1}$ and $(b_{i,j})_{i\ge 0, j\ge 1}$ such that for $1\le i\le 5$,
\begin{gather*}
U(X^i)=\sum_{j\ge 1}a(i,j)X^j,\\
\intertext{and for $0\le i\le 4$,}
U(\xi X^i)=\sum_{j\ge 1}b(i,j)X^j.
\end{gather*}
Both matrices $(a_{i,j})$ (for $i\ge 6$) and $(b_{i,j})$ (for $i\ge 5$) satisfy the same recurrence relation (here $m$ stands for a matrix):
\begin{align*}
m(i,j)&=55m(i-1,j-1)-300m(i-1,j-2)+875m(i-1,j-3)\\
&\quad\quad-1250m(i-1,j-4)+625m(i-1,j-5)\\
&\quad-60m(i-2,j-1)+175m(i-2,j-2)-250m(i-2,j-3)\\
&\quad\quad+125m(i-2,j-4)\\
&\quad+35m(i-3,j-1)-50m(i-3,j-2)+25m(i-3,j-3)\\
&\quad-10m(i-4,j-1)+5m(i-4,j-2)\\
&\quad+m(i-5,j-1),
\end{align*}
in which the undefined entries are assumed to be $0$.

\begin{theorem}
	We have, for $i\ge 1$,
	\begin{gather}
	U(X^i)=\sum_{j\ge 1}a(i,j)X^j,\label{eq:rec:a}\\
	\intertext{and for $i\ge 0$,}
	U(\xi X^i)=\sum_{j\ge 1}b(i,j)X^j.\label{eq:rec:b}
	\end{gather}
\end{theorem}

\begin{proof}
	It suffices to prove \eqref{eq:rec:a} for $i\ge 6$ and \eqref{eq:rec:b} for $i\ge 5$. Let $\zeta=e^{2\pi i/5}$. We define
	$$x_t=X(\zeta^t q)\quad \text{for $t=1,\ldots,5$}.$$
	Further, let $\sigma_t$ be the $t$-th elementary symmetric function of $x_1,\ldots,x_5$, that is,
	$$\sigma_t=\sum_{1\le k_1<k_2<\cdots<k_t\le 5}x_{k_1}x_{k_2}\cdots x_{k_t}.$$	
	For $i\ge 1$, let
	$$p_i=\sum_{t=1}^5 x_t^i=5U(X^i).$$
	Note that $p_i$ is obtained above for $i=1,\ldots, 5$. Therefore, we deduce from Newton's identities (cf.~\cite{Mea1992}) that
	\begin{align*}
	\sigma_1&=p_1\\
	&=55 X -300 X^2 + 875 X^3 -1250 X^4 + 625 X^5,\\
	\sigma_2&=(\sigma_1 p_1-p_2)/2\\
	&=60 X - 175 X^2+250 X^3 - 125 X^4,\\
	\sigma_3&=(\sigma_2 p_1-\sigma_1 p_2+p_3)/3\\
	&=35 X -50 X^2 + 25 X^3,\\
	\sigma_4&=(\sigma_3 p_1-\sigma_2 p_2+\sigma_1 p_3-p_4)/4\\
	&=10 X - 5 X^2,\\
	\sigma_5&=(\sigma_4 p_1-\sigma_3 p_2+\sigma_2 p_3-\sigma_1 p_4+p_5)/5\\
	&=X.
	\end{align*}
	
	Let $u$ be a complex-valued function and write $u_t = u(\zeta^t q)$ for $t=1,\ldots,5$. Since each $x_t$ ($t=1,\ldots,5$) satisfies
	$$x_t^5-\sigma_1 x_t^4+\sigma_2 x_t^3-\sigma_3 x_t^2+\sigma_4 x_t-\sigma_5=0,$$
	we conclude that, for $i\ge 6$,
	\begin{align*}
	5U(u X^i)&=\sum_{t=1}^5 u_t x_t^i=\sum_{t=1}^5 u_t (\sigma_1 x_t^{i-1}-\sigma_2 x_t^{i-2}+\sigma_3 x_t^{i-3}-\sigma_4 x_t^{i-4}+\sigma_5 x_t^{i-5})\\
	&= 5U(u X^{i-1}) \sigma_1- 5U(u X^{i-2}) \sigma_2+ 5U(u X^{i-3})\sigma_3\\
	&\quad - 5U(u X^{i-4})\sigma_4+ 5U(u X^{i-5})\sigma_5.
	\end{align*}
	Finally, \eqref{eq:rec:a} and \eqref{eq:rec:b} follow by taking $u=1$ and $u=\xi X^{-1}$, respectively.
\end{proof}

\subsection{Dissections}

Let us define a family of sequences $(d_\alpha)_{\alpha\ge 1}$ with $d_1$ being
$$(5,\,0,\,0,\,\ldots),$$
and the remaining sequences being recursively defined for $\alpha\ge 2$ by
\begin{align*}
d_{\alpha}(j)=\begin{cases}
\displaystyle\sum_{k\ge 1} a(k,j)d_{\alpha-1}(k) & \text{if $\alpha$ is even},\\[15pt]
\displaystyle\sum_{k\ge 1} b(k,j)d_{\alpha-1}(k) & \text{if $\alpha$ is odd}.
\end{cases}
\end{align*}
Here we write the $j$-th element in $d_\alpha$ as $d_\alpha(j)$.

\begin{remark}
	It is not hard to see that the two infinite matrices $(a_{i,j})$ and $(b_{i,j})$ are row and column finite. We now define $(t_{i,j})_{i, j\ge 1}$ by
	$$t(i,j)=\sum_{k\ge 1} a(i,k)b(k,j).$$
	Note that this sum is indeed a finite sum for fixed $i$ and $j$. Also, it follows from the recurrence relation of $d_\alpha$ that for $\alpha\ge 1$,
	$$d_{2\alpha+1}(j)=\sum_{i\ge 1}t(i,j)d_{2\alpha-1}(i).$$
\end{remark}

Recall that
$$
\sum_{n\ge 0}g(n)q^n=\frac{E(q^2)^3}{E(q)^2}.
$$

\begin{theorem}\label{th:g-gf}
	For $\alpha\ge 1$, we have
	\begin{gather}
	\sum_{n\ge 0} g\left(5^{2\alpha-1}n+\frac{5^{2\alpha}-1}{6}\right)q^n=\frac{E(q^{10})^3}{E(q^5)^2}\sum_{j\ge 1}d_{2\alpha-1}(j)X^j\\
	\intertext{and}
	\sum_{n\ge 0} g\left(5^{2\alpha}n+\frac{5^{2\alpha}-1}{6}\right) q^n=\frac{E(q^{2})^3}{E(q)^2}\sum_{j\ge 1}d_{2\alpha}(j)X^j.
	\end{gather}
\end{theorem}

\begin{proof}
	We know from \eqref{eq:U-xi} that
	\begin{align*}
	5X&=U(\xi)=U\Bigg(q^{-4}\frac{E(q^2)^3 E(q^{25})^2}{E(q)^2 E(q^{50})^3}\Bigg)=U\Bigg(\frac{E(q^{25})^2}{E(q^{50})^3}\sum_{n\ge 0} g(n) q^{n-4}\Bigg)\\
	&=\frac{E(q^5)^2}{E(q^{10})^3}\sum_{n\ge 0} g(5n+4) q^{n}.
	\end{align*}
	This proves the case $\alpha=1$. Assuming that the theorem is true for some positive odd integer $2\alpha-1$, namely,
	$$\sum_{n\ge 0} g\left(5^{2\alpha-1}n+\frac{5^{2\alpha}-1}{6}\right)q^n=\frac{E(q^{10})^3}{E(q^5)^2}\sum_{k\ge 1}d_{2\alpha-1}(k)X^k,$$
	we apply the $U$-operator to both sides and deduce
	\begin{align*}
	\sum_{n\ge 0} g\left(5^{2\alpha}n+\frac{5^{2\alpha}-1}{6}\right) q^n&=\frac{E(q^{2})^3}{E(q)^2}\sum_{k\ge 1}d_{2\alpha-1}(k)U(X^k)\\
	&=\frac{E(q^{2})^3}{E(q)^2}\sum_{k\ge 1} d_{2\alpha-1}(k) \sum_{j\ge 1}a(k,j)X^j\\
	&=\frac{E(q^{2})^3}{E(q)^2}\sum_{j\ge 1}\left(\sum_{k\ge 1}a(k,j)d_{2\alpha-1}(k)\right)X^j\\
	&=\frac{E(q^{2})^3}{E(q)^2}\sum_{j\ge 1}d_{2\alpha}(j)X^j.
	\end{align*}
	Further, multiplying both sides of the above identity by $q^{-4} E(q^{25})^2/E(q^{50})^3$ gives
	$$\frac{E(q^{25})^2}{E(q^{50})^3}\sum_{n\ge 0} g\left(5^{2\alpha}n+\frac{5^{2\alpha}-1}{6}\right) q^{n-4}=\sum_{k\ge 1}d_{2\alpha}(k)\xi X^k.$$
	It follows by applying the $U$-operator once again that
	\begin{align*}
	\sum_{n\ge 0} g\left(5^{2\alpha+1}n+\frac{5^{2\alpha+2}-1}{6}\right)q^n&=\frac{E(q^{10})^3}{E(q^5)^2}\sum_{k\ge 1}d_{2\alpha}(k)U(\xi X^k)\\
	&=\frac{E(q^{10})^3}{E(q^5)^2}\sum_{k\ge 1}d_{2\alpha}(k) \sum_{j\ge 1}b(k,j)X^j\\
	&=\frac{E(q^{10})^3}{E(q^5)^2}\sum_{j\ge 1}\left(\sum_{k\ge 1} b(k,j)d_{2\alpha}(k)\right) X^j\\
	&=\frac{E(q^{10})^3}{E(q^5)^2}\sum_{j\ge 1}d_{2\alpha+1}(j)X^j.
	\end{align*}
	We therefore arrive at the desired result by induction.
\end{proof}

\section{$5$-Adic orders}

For any integer $n$, let $\pi(n)$ be the $5$-adic order of $n$ with the convention that $\pi(0)=\infty$. Let $\fl{x}$ be the largest integer not exceeding $x$.

\subsection{Matrices $(a_{i,j})$ and $(b_{i,j})$}

We now determine the $5$-adic orders of $(a_{i,j})$ and $(b_{i,j})$.

\begin{theorem}
	For positive integers $i$ and $j$, we have
	\begin{gather}
	\pi(a(i,j))\ge \fl{\frac{5j-i-1}{6}}\label{ineq:a}\\
	\intertext{and}
	\pi(b(i,j))\ge \fl{\frac{5j-i+2}{6}}.\label{ineq:b}
	\end{gather}
\end{theorem}

\begin{proof}
	Since the first five rows of $(a_{i,j})$ are given, one may directly check \eqref{ineq:a} for $1\le i\le 5$. Assume that \eqref{ineq:a} holds for $1,\ldots,i-1$ with some $i\ge 6$. We have (for some undefined entries like $a(i,0)$, since we assign its value to be $0$, its $5$-adic order is therefore $\infty$):
	\begin{align*}
	\pi(a(i-1,j-1))+1\ge \fl{\dfrac{5(j-1)-(i-1)-1}{6}}+1&\ge \fl{\dfrac{5j-i-1}{6}},\\
	\pi(a(i-1,j-2))+2\ge \fl{\dfrac{5(j-2)-(i-1)-1}{6}}+2&\ge \fl{\dfrac{5j-i-1}{6}},\\
	\pi(a(i-1,j-3))+3\ge \fl{\dfrac{5(j-3)-(i-1)-1}{6}}+3&\ge \fl{\dfrac{5j-i-1}{6}},\\
	\pi(a(i-1,j-4))+4\ge \fl{\dfrac{5(j-4)-(i-1)-1}{6}}+4&\ge \fl{\dfrac{5j-i-1}{6}},\\
	\pi(a(i-1,j-5))+4\ge \fl{\dfrac{5(j-5)-(i-1)-1}{6}}+4&\ge \fl{\dfrac{5j-i-1}{6}},\\
	\pi(a(i-2,j-1))+1\ge \fl{\dfrac{5(j-1)-(i-2)-1}{6}}+1&\ge \fl{\dfrac{5j-i-1}{6}},\\
	\pi(a(i-2,j-2))+2\ge \fl{\dfrac{5(j-2)-(i-2)-1}{6}}+2&\ge \fl{\dfrac{5j-i-1}{6}},\\
	\pi(a(i-2,j-3))+3\ge \fl{\dfrac{5(j-3)-(i-2)-1}{6}}+3&\ge \fl{\dfrac{5j-i-1}{6}},\\
	\pi(a(i-2,j-4))+3\ge \fl{\dfrac{5(j-4)-(i-2)-1}{6}}+3&\ge \fl{\dfrac{5j-i-1}{6}},\\
	\pi(a(i-3,j-1))+1\ge \fl{\dfrac{5(j-1)-(i-3)-1}{6}}+1&\ge \fl{\dfrac{5j-i-1}{6}},\\
	\pi(a(i-3,j-2))+2\ge \fl{\dfrac{5(j-2)-(i-3)-1}{6}}+2&\ge \fl{\dfrac{5j-i-1}{6}},\\
	\pi(a(i-3,j-3))+2\ge \fl{\dfrac{5(j-3)-(i-3)-1}{6}}+2&\ge \fl{\dfrac{5j-i-1}{6}},\\
	\pi(a(i-4,j-1))+1\ge \fl{\dfrac{5(j-1)-(i-4)-1}{6}}+1&\ge \fl{\dfrac{5j-i-1}{6}},\\
	\pi(a(i-4,j-2))+1\ge \fl{\dfrac{5(j-2)-(i-4)-1}{6}}+1&\ge \fl{\dfrac{5j-i-1}{6}},\\
	\pi(a(i-5,j-1))+0\ge \fl{\dfrac{5(j-1)-(i-5)-1}{6}}+0&\ge \fl{\dfrac{5j-i-1}{6}}.
	\end{align*}
	It follows from the recurrence relation of $(a_{i,j})$ that $\pi(a(i,j))$ is at least the minimum of the left-hand sides of the above 15 inequalities. We therefore finish the proof of \eqref{ineq:a} by induction.
	
	Note that $(b_{i,j})$ shares the same recurrence relation with $(a_{i,j})$. Therefore, \eqref{ineq:b} can be shown analogously and the details are omitted.
\end{proof}

From the definition of $(t_{i,j})$, we have
\begin{align}\label{ineq:t}
\pi(t(i,j))&\ge \min_{k\ge 1}\left\{\pi(a(i,k))+\pi(b(k,j))\right\}\notag\\
&\ge \min_{k\ge 1}\left\{\fl{\frac{5k-i-1}{6}}+\fl{\frac{5j-k+2}{6}}\right\}.
\end{align}

\subsection{Sequences $d_{2\alpha-1}$}

We next study the $5$-adic order of $d_{2\alpha-1}(j)$ for positive integers $\alpha$ and $j$.

\begin{theorem}\label{th:5-adic-d}
	For positive integers $\alpha$ and $j$, we have
	\begin{align}\label{ineq:d}
	\pi(d_{2\alpha-1}(j))\ge \alpha+\fl{\frac{5j-5}{6}}.
	\end{align}
\end{theorem}

\begin{proof}
	It follows from $d_1=(5,0,0,\ldots)$ that \eqref{ineq:d} is true for $\alpha=1$. Assuming that the theorem holds for some positive $\alpha$, we now show
	$$\pi(d_{2\alpha+1}(j))\ge \alpha+1+\fl{\frac{5j-5}{6}}$$
	for all $j\ge 1$ and therefore conclude the theorem by induction.
	
	Note that for $\alpha\ge 1$, we have
	$$d_{2\alpha+1}(j)=\sum_{i\ge 1}t(i,j)d_{2\alpha-1}(i).$$
	Hence,
	$$\pi(d_{2\alpha+1}(j))\ge \min_{i\ge 1}\left\{\pi(d_{2\alpha-1}(i))+\pi(t(i,j))\right\}.$$
	It follows from \eqref{ineq:t} that
	\begin{align*}
	\pi(d_{2\alpha-1}(i))+\pi(t(i,j))&\ge \alpha+\fl{\frac{5i-5}{6}}+\min_{k\ge 1}\left\{\fl{\frac{5k-i-1}{6}}+\fl{\frac{5j-k+2}{6}}\right\}\\
	&\ge \alpha+\fl{\frac{5i-5}{6}}+\min_{k\ge 1}\left\{\fl{\frac{5j-i+4k-5}{6}}\right\}\\
	&=\alpha+\fl{\frac{5i-5}{6}}+\fl{\frac{5j-i-1}{6}}\\
	&\ge \alpha+1+\fl{\frac{5j-5}{6}}+\fl{\frac{4i-13}{6}}.
	\end{align*}
	Observing that $4i-13> 0$ for $i\ge 4$, it turns out that we only need to examine the cases $1\le i\le 3$.
	
	For $i=3$, we have
	\begin{align*}
	\pi(d_{2\alpha-1}(3))+\pi(t(3,j))&\ge \alpha+1+\min_{k\ge 1}\left\{\fl{\frac{5k-3-1}{6}}+\fl{\frac{5j-k+2}{6}}\right\}\\
	&\ge \alpha+1+\fl{\frac{5j-4}{6}}\ge \alpha+1+\fl{\frac{5j-5}{6}}.
	\end{align*}
	
	For $i=1$ and $2$, we have
	\begin{align*}
	\pi(d_{2\alpha-1}(i))+\pi(t(i,j))&\ge \alpha+\min_{k\ge 1}\left\{\fl{\frac{5k-i-1}{6}}+\fl{\frac{5j-k+2}{6}}\right\}.
	\end{align*}
	Further,
	\begin{align*}
	\fl{\frac{5k-i-1}{6}}+\fl{\frac{5j-k+2}{6}}&\ge
	\begin{cases}
	\fl{\frac{5j+1}{6}} & \text{if $k=1$}\\[5pt]
	\fl{\frac{5j+4k-7}{6}} & \text{if $k\ge2$}
	\end{cases}\\
	&\ge \fl{\frac{5j-5}{6}}+1.
	\end{align*}
	
	We therefore have, for all $j\ge 1$,
	$$\pi(d_{2\alpha+1}(j))\ge \alpha+1+\fl{\frac{5j-5}{6}},$$
	which is our desired inequality.
\end{proof}

\section{Congruences for 1-shell TSPPs}\label{Sect:cong}

\subsection{Proof of the main theorem}

Now we are in the position of proving Theorem \ref{th:main}. By Theorems \ref{th:g-gf} and \ref{th:5-adic-d}, we have
\begin{align}\label{eq:g-mod-powers-5}
\sum_{n\ge 0} g\left(5^{2\alpha-1}n+\frac{5^{2\alpha}-1}{6}\right)q^n&=\frac{E(q^{10})^3}{E(q^5)^2}\sum_{j\ge 1}d_{2\alpha-1}(j)X^j\notag\\
&\equiv 0 \pmod{5^\alpha}.
\end{align}

To show
$$s\left(2\cdot 5^{2\alpha-1}n+5^{2\alpha-1}\right)\equiv 0 \pmod{5^{\alpha}},$$
we first notice that
$$s(2\cdot 5^{2\alpha-1}(3n)+5^{2\alpha-1})=s(6\cdot 5^{2\alpha-1}n+5^{2\alpha-1})=0$$
since $6\cdot 5^{2\alpha-1}n+5^{2\alpha-1} \equiv 2 \pmod{3}$ and
$$s(2\cdot 5^{2\alpha-1}(3n+1)+5^{2\alpha-1})=s(6\cdot 5^{2\alpha-1}n+3\cdot 5^{2\alpha-1})=0$$
since $6\cdot 5^{2\alpha-1}n+3\cdot 5^{2\alpha-1} \equiv 0 \pmod{3}$. Hence, we only need to examine
$$s\left(2\cdot 5^{2\alpha-1}(3n+2)+5^{2\alpha-1}\right)=s\left(6\cdot 5^{2\alpha-1}n+5^{2\alpha}\right)\equiv 0 \pmod{5^\alpha}.$$
Finally, it follows from \eqref{eq:fg} and \eqref{eq:g-mod-powers-5} that
\begin{align*}
s\left(6\cdot 5^{2\alpha-1}n+5^{2\alpha}\right)&=s\left(6\left(5^{2\alpha-1}n+\frac{5^{2\alpha}-1}{6}\right)+1\right)\\
&=g\left(5^{2\alpha-1}n+\frac{5^{2\alpha}-1}{6}\right)\equiv 0 \pmod{5^{\alpha}}.
\end{align*}

\subsection{More congruences}\label{Sect:more-cong}

We begin with one of Ramanujan's theta functions
$$\phi(q):=\sum_{n=-\infty}^\infty q^{n^2}.$$
It is known that
$$\phi(-q)=\frac{E(q)^2}{E(q^2)}.$$

Ramanujan obtained the 5-dissection formula of $\phi(-q)$ (cf.~\cite[(36.3.2) with $q$ replaced by $-q$]{Hir2017}) as follows.
\begin{lemma}\label{le:phi5}
	We have
	\begin{equation}
	\phi(-q)=\phi(-q^{25})-2qM_1(-q^{5})+2q^4M_2(-q^{5}),
	\end{equation}
	where $M_1(-q)=(q^3,q^7,q^{10};q^{10})_\infty$ and $M_2(-q)=(q,q^9,q^{10};q^{10})_\infty$.
\end{lemma}

We also have (cf.~\cite[(36.3.4) with $q^5$ replaced by $-q$]{Hir2017}) the following relation between $M_1(-q)$ and $M_2(-q)$.
\begin{lemma}\label{le:m1m2}
	We have
	\begin{equation}
	4qM_1(-q)M_2(-q)=\phi(-q^5)^2-\phi(-q)^2.
	\end{equation}
\end{lemma}

Observing that $d_{2\alpha-1}(j)\ge \alpha+\fl{(5j-5)/6}\ge \alpha+1$ for $j\ge 3$, one has
\begin{align*}
&\sum_{n\ge 0} g\left(5^{2\alpha-1}n+\frac{5^{2\alpha}-1}{6}\right)q^n\\
&\qquad\qquad\equiv \frac{E(q^{10})^3}{E(q^5)^2} \left(d_{2\alpha-1}(1)X+d_{2\alpha-1}(2)X^2\right) \pmod{5^{\alpha+1}}.
\end{align*}
We further notice that
$$X=\frac{E(q^2)^2 E(q^5)^4}{E(q)^4 E(q^{10})^2}\equiv \left(\frac{E(q)^2}{E(q^2)}\right)^8=\phi(-q)^8 \pmod{5}.$$

\medskip

One may compute that
\begin{gather*}
d_3(1)\equiv 3\cdot 5^2 \pmod{5^3}\\
\intertext{and}
d_3(2)\equiv  5^2 \pmod{5^3}.
\end{gather*}
Hence,
\begin{align*}
&\sum_{n\ge 0} g\left(125n+104\right)q^n\\
&\qquad\equiv 25\frac{E(q^{10})^3}{E(q^5)^2} \left(3\phi(-q)^8+\phi(-q)^{16}\right)\\
&\qquad\equiv25\frac{E(q^{10})^3}{E(q^5)^2} \left(3\phi(-q)^3\phi(-q^5)+\phi(-q)\phi(-q^5)^3\right) \pmod{125}.
\end{align*}
We now apply Lemmas \ref{le:phi5} and \ref{le:m1m2} to obtain
\begin{align*}
&\sum_{n\ge 0} g\left(625n+229\right)q^n= \sum_{n\ge 0} g\left(125(5n+1)+104\right)q^n\\
&\equiv 25\frac{E(q^{2})^3}{E(q)^2} \left(3\phi(-q)M_1(-q)\left(4\phi(-q^5)^2+4qM_1(-q)M_2(-q)\right)+3\phi(-q)^{3}M_1(-q)\right)\\
&\equiv 75M_1(-q)\frac{E(q^{2})^3}{E(q)^2}\left(4\phi(-q)\phi(-q^5)^2+\phi(-q)\left(\phi(-q^5)^2-\phi(-q)^2\right)+\phi(-q)^{3}\right)\\
&\equiv 0\pmod{125}.
\end{align*}
Likewise,
\begin{align*}
&\sum_{n\ge 0} g\left(625n+604\right)q^n= \sum_{n\ge 0} g\left(125(5n+4)+104\right)q^n\\
&\equiv 50M_2(-q)\frac{E(q^{2})^3}{E(q)^2}\left(4\phi(-q)\phi(-q^5)^2+\phi(-q)\left(\phi(-q^5)^2-\phi(-q)^2\right)+\phi(-q)^{3}\right)\\
&\equiv 0\pmod{125}.
\end{align*}
Consequently, we arrive at an alternative (and indeed elementary) proof of \eqref{eq:cong-chern}:
$$s(1250n+125,\ 1125)\equiv 0 \pmod{125},$$
since $s(n)=0$ for $n\ge 1$ with $n\equiv 0,2 \pmod{3}$, and $s(6n+1)=g(n)$ for $n\ge 0$.

\medskip

One may further compute that
\begin{gather*}
d_5(1)\equiv 0 \pmod{5^4}\\
\intertext{and}
d_5(2)\equiv  4\cdot 5^3 \pmod{5^4}.
\end{gather*}
Hence,
\begin{align*}
\sum_{n\ge 0} g\left(3125n+2604\right)q^n&\equiv 4\cdot 125\frac{E(q^{10})^3}{E(q^5)^2} \phi(-q)^{16}\\
&\equiv4\cdot 125\frac{E(q^{10})^3}{E(q^5)^2} \phi(-q)\phi(-q^5)^3  \pmod{625}.
\end{align*}
It follows from Lemma \ref{le:phi5} that $\phi(-q)$ contains no terms in which the power of $q$ is $2$ or $3$ modulo $5$. Hence,
\begin{gather*}
g(15625n+8854)=g(3125(5n+2)+2604)\equiv 0 \pmod{625}\\
\intertext{and}
g(15625n+11979)=g(3125(5n+3)+2604)\equiv 0 \pmod{625}.
\end{gather*}
Consequently, we have two congruences modulo $625$ as follows.
\begin{theorem}
	For $n\ge 0$, we have
	\begin{gather}
	s(31250n+9375)\equiv 0 \pmod{625}\\
	\intertext{and}
	s(31250n+21875)\equiv 0 \pmod{625}.
	\end{gather}
\end{theorem}

\subsection*{Acknowledgements}

I would like to acknowledge my gratitude to the referee for helpful comments on an earlier draft of this paper.

\bibliographystyle{amsplain}

\end{document}